\documentclass[lettertpaper, 10pt, twocolumn]{article}
\usepackage{cite}
\usepackage{amsmath,amssymb,amsfonts}
\usepackage{algorithmic}
\usepackage{graphicx}
\usepackage{textcomp}
\usepackage{xcolor}
\def\BibTeX{{\rm B\kern-.05em{\sc i\kern-.025em b}\kern-.08em
    T\kern-.1667em\lower.7ex\hbox{E}\kern-.125emX}}

\usepackage{amsthm}
\usepackage{mathrsfs}			
\usepackage{stmaryrd}

\newtheorem{theorem}{Theorem}[section]
\newtheorem{proposition}[theorem]{Proposition}
\newtheorem{lemma}[theorem]{Lemma}
\newtheorem{corollary}[theorem]{Corollary}

\newtheorem{definition}{Definition}[section]

\newtheorem{remark}{Remark}[section]

\newtheorem{example}{Example}[section]

\newcommand{\memo}[1]{}
\newcommand{\vs}[1]{\vspace{#1mm}}
\newcommand{\hs}[1]{\hspace{#1mm}}
\newcommand{\nt}{\notag}

\newcommand{\mcal}[1]{{\mathcal{#1}}}
\newcommand{\mbb}[1]{{\mathbb{#1}}}
\newcommand{\mfr}[1]{{\mathfrak{#1}}}
\newcommand{\mscr}[1]{{\mathscr{#1}}}
\newcommand{\pd}{\partial}
\newcommand{\vl}[1]{\text{\boldmath $#1$}}

\newenvironment{xyczero}
{\jot = 1mm \vs{-1}
\begin{array}{c}}{\end{array}
\vs{1} \jot = 0mm}

\usepackage[all]{xy}                

\def\Ker{\mathop{\mathrm{Ker}}\nolimits}
\def\Im{\mathop{\mathrm{Im}}\nolimits}

\def\Div{\mathop{\mathrm{Div}}\nolimits}
\def\Curl{\mathop{\mathrm{Curl}}\nolimits}
\def\Grad{\mathop{\mathrm{Grad}}\nolimits}

\newcommand{\Hodge}{\hs{0.2}\text{$\ast$}\hs{0.2}}

\def\VF{\mfr{X}}

\newcommand{\HK}{\VF_{HK}}

\newcommand{\HG}{\VF_{HG}}

\newcommand{\VG}{\VF_{G}}
\newcommand{\VK}{\VF_{K}}

\newcommand{\GV}{v}
\newcommand{\GVS}{v}
\newcommand{\GWS}{w}

\newcommand{\SE}{E}
\newcommand{\SD}{D}

\def\tangent{\vl{t}}
\def\normal{\vl{n}}

\newcommand{\ST}{T}
\newcommand{\SN}{N}

\newcommand{\HDR}{H_{DR}}
\newcommand{\HSD}{H_{SD}}

\newcommand{\HDRr}{H}
\newcommand{\HDRa}{H}

\newcommand{\HDRed}[2]{{\HDR^{#1}(#2, d)}}
\newcommand{\HDRde}[2]{{\HDR^{#1}(#2, \delta)}}

\newcommand{\Z}{M}

\def\rank{\mathop{\mathrm{rank}}\nolimits}		

\title{\LARGE \bf
Topological geometric extension of Stokes-Dirac structures \\ for global energy flows*
}

\author{Gou Nishida\thanks{This work was supported by JSPS KAKENHI Grant Number JP18K04213.}~\thanks{Gou Nishida is with Department of Electrical and Electronic Engineering, College of Engineering, Nihon University, 1 Nakagawara, Tokusada, Tamura, Koriyama, Fukushima, 963-8642, JAPAN. {\ttfamily \small g.nishida@ieee.org}}, and Bernhard Maschke\thanks{Bernhard Maschke is with Laboratoire d'Automatique et de G\'{e}nie des Proc\'{e}d\'{e}s, LAGEP UMR CNRS 5007, UFR Facult\'{e} des Sciences et Technologies, Universit\'{e} Claude Bernard Lyon 1, 43, bd du 11 Novembre 1918, F-69622 Villeurbanne cedex, FRANCE. {\ttfamily \small bernhard.maschke@univ-lyon1.fr}}%
}

\begin{document}

\maketitle

\begin{abstract}
This paper clarifies a global structure of Stokes-Dirac structures used for describing interconnected port-Hamiltonian systems defined on manifolds with non-trivial topology under consistent boundary condition.
\end{abstract}




\section{Introduction}\label{Sec1}

\textit{Port-Hamilton systems}~\cite{b_ASch1} have been developed as an extended Hamiltonian system, and it is one of most essential system representations for controlling complex physical systems.
Port-Hamilton systems are defined by particular variable pairs of an input and an output.
The pair is called \textit{port}, and the variables are called \textit{effort} and \textit{flow}.
The product of an effort and the corresponding flow has the physical dimension of power, i.e., the time derivative of energy.
Thus, the sum of the products of all pairs is equivalent to the time derivative of the total energy of a given system, i.e., \textit{Hamiltonian}.
Indeed, ports can be derived from derivatives of a Hamiltonian.
By using port-Hamiltonian systems, physical network systems interconnected through the ports can be described.
Then, an energy balance equation can be defined on the terminal of the port-interconnections by the port variables.
If the energy balance holds in any interaction with environments, this property can be used for stability analysis, because an energy of systems can be considered as a Lyapunov function.
Thus, various control methods using energy flows on such a network, e.g., passivity-based controls~\cite{b_ASch1, b_AMac1}, can be used.

Port-Hamilton systems have been extended for systems governed by partial differential equations, and it is called a \textit{distributed port-Hamilton system}~\cite{j_ASch1, b_AMac1}.
In the system, ports are defined on a boundary of a control system domain, and they are called \textit{boundary ports}.
Then, the balance equation can be augmented as that on the boundary.
As a remarkable property of distributed port-Hamilton systems, Stokes theorem can applied to the boundary energy balance, i.e., an internal energy variation can be transformed to a boundary energy variation.
Hence, a kind of boundary energy controls can be realized.
The boundary integrability can be formulated by a particular case of Dirac structures called \textit{the Stoke-Dirac structure}~\cite{j_ASch1, j_GNis1}, where the Dirac structure is a generalized concept of Poisson and Symplectic structures.

The port interconnection of distributed port-Hamiltonian systems determines a network of energy flows between systems domains through their boundaries.
The union of the domains with boundaries configures a particular shape; therefore, the energy flow can be considered as a vector field on a manifold if they are continuous.
Manifolds may have various topologies, and they can be classified and characterized by topological geometric concepts, homology and cohomology groups~\cite{b_SMor1, b_GSch1}.
The homology of manifolds is introduced from a triangulation of manifolds that is a decomposition consisting of fundamental figures with an integer dimension, called \textit{the simplicial complex}~\cite{b_SMor1}, and it means the number of closed loops that never divide a manifold into two disjoint subdomains in the two-dimensional case.
The cohomology of manifolds is a dual concept of the homology, and it is equivalent to the difference between two types of differential forms, i.e., closed and exact forms in \textit{the de Rham complex}~\cite{b_GSch1} that consists of space of differential forms and exterior derivatives.
Differential forms are used for defining of integrands over higher-dimensional surfaces, i.e., manifolds~\cite{b_SMor1}, and they independent of coordinates in multivariable calculus.
In~\cite{j_ASch2, j_PKot1, j_PKot2}, port-representations of discrete systems have been studied in terms of homology.
Cohomological approaches for port-representations were proposed in~\cite{c_GNis1, b_GNis1}.

The purpose of this paper is to clarify the relationship between the Stokes-Dirac structure and the topological geometry of the manifold on which the internal and boundary energy flows of distributed port-Hamiltonian systems defined.
The Stoke-Dirac structure is defined by differential forms, and a distributed port-Hamiltonian system with the Stokes-Dirac structure is actually defined on a domain of a manifold.
Therefore, differential forms used for the Stoke-Dirac structure must be affected by the shape of the manifold.
The original Stokes-Dirac structure has been defined on a domain with a trivial topology, i.e., it is contractible to a point.

This paper is organized in the following sections.
In the second section, the definition of manifolds with boundary, differential forms on the manifolds are introduced from a typical notation in topological geometry~\cite{b_SMor1, b_GSch1}.
The third section explains that, from these definitions, \textit{the Hodge-Morrey decomposition}~\cite{b_GSch1}, which is an extended \textit{Hodge-Kodaira decomposition}~\cite{b_SMor1} for closed manifolds, can be considered on manifolds with boundary.
Harmonic differential forms~\cite{b_SMor1, b_GSch1} must be considered on manifolds with non-trivial topology, and the topology of the manifolds can be characterized by the harmonic forms given by the Hodge-Morrey decomposition.
Moreover, the harmonic forms can be classified as tangent or orthogonal by Friedrichs decomposition~\cite{b_GSch1}.
As a result, the fourth section shows that an essential property of an extended Stokes-Dirac structure for defining distributed port-Hamiltonian systems on manifolds with non-trivial topology can be derived from our previous results~\cite{j_GNis1,b_GNis1}.


\section{Mathematical preliminary}\label{Sec2}

The mathematical notation of this paper follows the reference~\cite{b_GSch1}.
Some basic concepts are explained in~\cite{b_SMor1}.


\subsection{Manifold with boundary}

Let $M$ be a paracompact topological Hausdorff space, where $M$ is called paracompact if any open covering $\{U_a\}_{a \in A}$ of $M$ can be refined by locally finite covering, and a covering $U$ is called the refinement of a covering $V$ if any set in $U$ is included in some set in $V$.

A surjective homeomorphism $\varphi_a\colon U_a \to \mbb{R}^n_{\vl{u}_a}$ to an open subset in $\mbb{R}^n_{\vl{u}_a}$ for $a \in A$ is called a \emph{chart}, where we have defined a real half space $\mbb{R}^n_{\vl{u}} = \{ \vl{x} \in \mbb{R}^n \mid \langle \vl{x}, \vl{u} \rangle \geq 0 \}$ for a fixed vector $\vl{u} \neq 0$ in $\mbb{R}^n$.
Hence, $\mcal{A}_M = \{(U_a, \varphi_a)\}_{a \in A}$ becomes an \emph{atlas} on $M$.

Then, \emph{the boundary} of $M$ is defined by
\begin{align}
 \pd M &:= \left\{ p \in M \mid \exists \varphi_a \textrm{\ s.t.\ } \langle \varphi_a(p), \vl{u}_a \rangle = 0 \right. \nt \\
&\hs{40}\left. \textrm{\ for some\ } \vl{u}_a \right\},
\end{align}
where $\vl{u}_a$ may be different for each $a \in A$.


\begin{definition}\label{PDmfd}
 We call a manifold $M$ with boundary \emph{$\pd$-manifold} if 
1) $M$ is equipped with a smooth atlas $\mcal{A}_M$, 
2) $\mcal{A}_M$ is oriented, 
3) $M$ is complete as a metric space.
\end{definition}


\subsection{Differential forms on boundary}

A differential $k$-form $\omega \in \varOmega^k(M)$ is an anti-symmetric $k$-linear map 
\begin{align}
 \omega \colon &\Gamma(TM) \times \stackrel{k}{\cdots} \times \Gamma(TM) \to C^\infty(M);\, \nt \\
 &(X_1, \cdots, X_k) \mapsto \omega(X_1, \cdots, X_k).
\end{align}
%

In the same way of the decomposition of vector fields (see \textit{Section\;\ref{VFonB}}) on $\pd M$, 
we can define 
\begin{align}
 \tangent \omega (X_1, \cdots, X_k) &= \omega (X_1^\parallel, \cdots, X_k^\parallel) \nt \\
 &\quad \forall X_1, \cdots, X_k \in \Gamma(TM|_{\pd M}), \\
 \normal \omega &= \omega|_{\pd M} - \tangent \omega
\end{align}
for $k \geq 1$, and $\tangent \omega = \omega$ for $k = 0$.


\begin{proposition}\label{Prop21}
The tangential component $\tangent \omega$ is uniquely determined by the pull-back $\iota^\ast\colon \varOmega^{k}(M) \to \varOmega^{k}(\pd M)$ of the inclusion $\iota\colon \pd M \to M$.
Thus, we get the identification $\iota^\ast \omega = \iota^\ast \tangent \omega = \tangent \omega$.
\end{proposition}


\begin{definition}\label{Def22}
 Let $M$ be $\pd$-manifold. We define the exterior product $\wedge\colon \varOmega^i(M) \times \varOmega^j(M) \to \varOmega^{i+j}(M)$, the Hodge star operator $\Hodge\colon \varOmega^i(M) \to \varOmega^{n-i}(M)$, the exterior derivative $d\colon \varOmega^i(M) \to \varOmega^{i+1}(M)$, and the co-differential operator $\delta\colon \varOmega^i(M) \to \varOmega^{i-1}(M);\ \omega \mapsto (-1)^{n i + n + 1}\Hodge d (\Hodge \omega)$ that is adjoint to $d$ with respect to $L^2$-metric $\left<\left< \omega, \eta \right>\right> = \int_M \omega \wedge \Hodge \eta$.
\end{definition}

\begin{proposition}\label{Prop1}
 Let $M$ be $\pd$-manifold, and let $\mcal{F}$ be a normal frame on $U \subset M$.
Then, the normal and tangential components of $\omega \in \varOmega^k(M)$ are adjoint to each other in the sense of Hodge operator $\Hodge\colon \varOmega^k(M) \to \varOmega^{n-k}(M)$, i.e., 
$\Hodge (\normal \omega) = \tangent (\Hodge \omega)$, and $\Hodge (\tangent \omega) = \normal (\Hodge \omega)$.
Moreover, $\tangent (d \omega) = d (\tangent \omega)$, and $\normal (\delta \omega) = \delta (\normal \omega)$ hold. 
\end{proposition}


\begin{definition}
 The Sobolev spaces $W^{0,p}\varOmega^k(M)$ and $W^{s,2}\varOmega^k(M)$ (see \textit{Section\;\ref{Sobolev}}) are denoted by $L^{p}\varOmega^k(M)$ and $H^{s}\varOmega^k(M)$, respectively.
\end{definition}


\section{Important known results for decompositions of differential forms}\label{Sec3}

Some useful results on the topological geometry of manifolds with boundary are introduced from the reference~\cite{b_GSch1}.


\subsection{Hodge decomposition on manifolds with boundary}

For differential forms on manifolds with boundary, the following well-known relation holds.


\begin{theorem}[Stokes theorem]\label{Stokes}
For all $\omega \in W^{1,1} \varOmega^{n-1}(M)$ on a $\pd$-manifold $M$ with compact boundary $\pd M$, 
 \begin{align}
  \int_M d\omega = \int_{\pd M} \iota^\ast \omega.
 \end{align}
\end{theorem}


\begin{theorem}[Green formula]\label{Green}
 Let us consider $\omega \in W^{1,p}\varOmega^{k-1}(M)$, and $\eta \in W^{1,q} \varOmega^{k}(M)$ on a $\pd$-manifold $M$, where $1/p + 1/q = 1$.
 Then,
 \begin{align}
  \int_M d\omega \wedge \Hodge \eta =  \int_M \omega \wedge \Hodge \delta \eta + \int_{\pd M} \tangent \omega \wedge \Hodge \normal \eta.
 \end{align}
\end{theorem}


\begin{definition}\label{DNfields}
For the space of $\omega \in H^1 \varOmega^k(M)$ on a $\pd$-manifold $M$, 
\emph{the space of harmonic fields} in $H^1 \varOmega^k(M)$ is defined by
 \begin{align}
  \mcal{H}^k(M) = \left\{ \lambda \in H^1 \varOmega^k(M) \mid d\lambda = 0, \delta \lambda = 0 \right\}.
 \end{align}
The spaces of $\omega$ that vanish tangential or normal components are defined by
 \begin{align}
  H^1 \varOmega^k_\ST(M) &= \left\{ \omega \in H^1 \varOmega^k(M) \mid \tangent \omega = 0 \right\}, \\
  H^1 \varOmega^k_\SN(M) &= \left\{ \omega \in H^1 \varOmega^k(M) \mid \normal \omega = 0 \right\}.
 \end{align}
 Then, the subspaces
 \begin{align}
  \mcal{H}^k_\ST(M) &= H^1 \varOmega^k_\ST(M) \cap \mcal{H}^k(M), \\
  \mcal{H}^k_\SN(M) &= H^1 \varOmega^k_\SN(M) \cap \mcal{H}^k(M).
 \end{align}
 are called \emph{Dirichlet and  Neumann fields}, respectively.
 If $\pd M = \emptyset$, $H^1 \varOmega^k_\ST(M) = H^1 \varOmega^k_\SN(M) = H^1 \varOmega^k(M)$.
\end{definition}


\begin{remark}
 By the Hodge duality in \textit{Proposition\;\ref{Prop1}}, $\mcal{H}^k_\ST(M) \cong \mcal{H}^{n-k}_\SN(M)$ holds.
\end{remark}


\begin{definition}
 Let $M$ be a $\pd$-manifold. The subspaces of exact and co-exact $k$-forms with vanishing tangential and normal component in $L^2 \varOmega^k(M)$ are defined by
 \begin{align}
  \mcal{E}^k(M) &= \left\{ d \alpha \mid \alpha \in H^1 \varOmega^{k-1}_\ST(M) \right\}, \\ 
  \mcal{C}^k(M) &= \left\{ \delta \beta \mid \beta \in H^1 \varOmega^{k+1}_\SN(M) \right\}, 
 \end{align}
 where $\mcal{E}^0(M) = \{ 0 \}$ and $\mcal{C}^n(M) = \{ 0 \}$.
\end{definition}


\begin{theorem}[Hodge-Morrey decomposition]\label{HMDcomp}
 The Hilbert space $L^2 \varOmega^k(M)$ of square integrable $k$-forms on a compact $\pd$-manifold $M$ can be split into the $L^2$-orthogonal direct sum 
 \begin{align}
  L^2 \varOmega^k(M) = \mcal{E}^k(M) \varoplus \mcal{C}^k(M) \varoplus L^2 \mcal{H}^k(M),
 \end{align}
 where $L^2 \mcal{H}^k(M)$ is the $L^2$-closure of the space of harmonic fields $\mcal{H}^k(M)$.
\memo{
 For {\color{red}$W^{s,p}\varOmega^k(M)$}, where $s \in \mbb{N}$ and $p \geq 2$, 
the following corresponding $L^2$-orthogonal decomposition holds:
 \begin{align}
  &W^{s,p}\varOmega^k(M) = \nt \\
  &\quad W^{s,p}\mcal{E}^k(M) \varoplus W^{s,p}\mcal{C}^k(M) \varoplus W^{s,p}\mcal{H}^k(M).
 \end{align}
} 
\end{theorem}

 When $\pd M = \emptyset$, the Hodge-Morrey decomposition includes the Hodge-Kodaira decomposition for compact manifolds. 


\begin{theorem}[Friedrichs decomposition]\label{FDcomp}
 The space $\mcal{H}^k(M) \subset H^1\varOmega^k(M)$ of harmonic fields on a compact $\pd$-manifold $M$ can be decomposed into 
 \begin{align}
  \mcal{H}^k(M) &= \mcal{H}^k_\ST(M) \varoplus \mcal{H}^k_{C}(M), \\
  \mcal{H}^k(M) &= \mcal{H}^k_\SN(M) \varoplus \mcal{H}^k_{E}(M),
 \end{align}
 where the subspaces of $\mcal{H}^k(M)$, i.e., the subspaces of exact harmonic and co-exact harmonic fields have been defined as follows:
 \begin{align}
  \mcal{H}^k_{E}(M) &:= \left\{ \kappa \in \mcal{H}^k(M) \mid \kappa = d \epsilon \right\},\\
  \mcal{H}^k_{C}(M) &:= \left\{ \kappa \in \mcal{H}^k(M) \mid \kappa = \delta \gamma \right\}.
 \end{align}
 These decompositions are valid for $W^{s,p}\mcal{H}^k(M)$, where $s \in \mbb{N}_0$, and $p \geq 2$.
\end{theorem}


\subsection{De Rham complex on manifolds with boundary}

The de Rham complex consists of the set of the space of differential forms and the exterior derivative.
The cohomology of the complex is related with the topology of manifolds through harmonic forms according to Hodge theorem.

This section explains the relationship between the de Rham complex and the previously discussed particular subspaces of harmonic forms on $\pd$-manifolds obeying boundary conditions.


\begin{definition}
 A form $\omega \in \varOmega^k(M)$ is called \emph{closed} if $d \omega = 0$.
 A form $\omega \in \varOmega^k(M)$ is called \emph{exact} if there exist a form $\eta \in \varOmega^{k-1}(M)$ such that $d \eta = \omega$.
\end{definition}


\begin{definition}
 Let $M$ be a $\pd$-manifold.
 \emph{The $k$th cohomology of the de Rham complex} $(\varOmega(M),d)$ of differential forms over $M$ without imposing boundary conditions is the quotient space
 \begin{align}
  \HDR^k(M,d) = \Ker d^k/\Im d^{k-1},
 \end{align}
 where the exterior derivative $d$ for $k$-forms is denoted by $d^k\colon \varOmega^k(M) \to \varOmega^{k+1}(M)$, and the spaces of all closed $k$-forms and all exact $k$-forms are defined by
the cycle $Z^k(M) = \Ker d^k$ and the boundary $B^k(M) = \Im d^{k-1}$, respectively.
 \emph{The $k$th cohomology of the dual complex} $(\varOmega(M),\delta)$ is
 \begin{align}
  \HDR^k(M,\delta) = \Ker \delta^k/\Im \delta^{k+1},
 \end{align}
 where $\delta^k\colon \varOmega^k(M) \to \varOmega^{k-1}(M)$ is the co-differential operator. 
\end{definition}


\begin{theorem}[Hodge isomorphism]\label{HodgeIso}
 Let $M$ be a compact $\pd$-manifold. Then,
 \begin{align}
  \HDR^k(M,d) \cong \mcal{H}^k_\SN(M), \quad \HDR^k(M,\delta) \cong \mcal{H}^k_\ST(M).
 \end{align}
\end{theorem}


\begin{corollary}\label{HodgeIsoCor1}
 On a compact $\pd$-manifold $M$,
 \begin{align}
  \HDR^k(M,d) \cong \HDR^{n-k}(M,\delta).
 \end{align}
 %
\end{corollary}

\memo{
\begin{corollary}
The Hodge operator $\Hodge$ on $\omega(M)$ induces an isomorphism
 \begin{align}
  \Hodge_P\colon \HDRa^k(M) \to \HDRr^{n-k}(M,\pd M)
 \end{align}
 that is referred to as the Poincar\'{e} duality for manifolds with boundary.
\end{corollary}
} 


\section{Main results}\label{Sec4}

\memo{
{\color{blue}

\begin{enumerate}
\item Hodge decomposition
\item Harmonic form is an extra term in standard Stokes-Dirac structure
\item What is the physical meaning of Harmonic forms?
\item Harmonic forms are related to global topology
\item That is used for dividing and interconnecting of global topology in the sense of port-representations.
\end{enumerate}

}
} 


\subsection{Extension of Stokes-Dirac structures in terms of topological geometry}

The port-Hamiltonian representation of distributed parameter systems is derived from the Stokes-Dirac structure~\cite{j_ASch1}.
The Stokes-Dirac structure is one of Dirac structures that is a generalized symplectic and poisson, and it inherits boundary integrability from the Stokes theorem~\cite{b_SMor1} in terms of differential forms defined on manifolds that is an abstract version of that in 3-dimensional vector analysis.
Therefore, the Stokes-Dirac structure clarifies the relation between variables distributed in the internal of a system domain and variables restricted on its boundary in the sense of boundary integration.
Thus, a power balance equation on a boundary of the system can be introduced from the relation, and it can be used for boundary energy controls of distributed parameter systems~\cite{b_AMac1}.


In this methodology, the domain is assumed to be contractible, i.e., it is continuously shrunk to a point just like a Euclidian space 
Then, the boundary is considered as smooth regions surrounding such a domain.
In other words, such a domain with the boundary is called homeomorphic to a point that means the simplest case of the shape of manifolds, so to say, topologically trivial (e.g., a donut or a cup with a handle is not trivial).
In general, the shapes of system domains and their boundaries are complex.
One of understandable examples is an electrical circuit that may include non-contractible loops by regarding electrical elements and wirings as 1-dimensional domains.
In such a case, constraints that are Kirchhoff's laws must be additionally considered for describing interconnections.


This paper attempts to extend the domain of Stokes-Dirac structures to be more general shape.
Then, it is clarified what conditions dynamics defined on domains with a non-trivial topology must satisfy in their port-representations and interconnections.
In such an extension, a particular differential form, called a Harmonic plays a central role, and the form is closely related with the topology of manifolds. 
Simply speaking, harmonic forms reflect the shape of manifolds through their (co)homology groups (see also \textit{Remark\;\ref{Rem41}} for details).
A harmonic form $\omega$ is defined by $d \omega = 0$ and $\delta \omega = 0$ (see \textit{Definition\;\ref{Def22}}), and it is included in differential forms defined on manifolds with a non-trivial topology.
We shall first see a more basic result than \textit{Theorem\;\ref{HMDcomp}} for manifolds with boundary. 
Indeed, according to the Hodge decomposition on a compact manifold without boundary, an arbitrary $k$-form $\omega \in \varOmega^k(M)$ can be written as
\begin{align}
 \omega = d \alpha + \delta \beta + \omega_H
\end{align}
where $\alpha \in \varOmega^{k-1}(M)$, $\beta \in \varOmega^{k+1}(M)$, and $(d \delta + \delta d)\, \omega_H = 0$.
Harmonic forms are a generalization of harmonic functions that means solutions of Poisson equations at equilibrium in, e.g., eigenvalue problems of elliptic partial differential equations.
Actually, $\varDelta = d \delta + \delta d$ is called a Laplacian.
On the other hand, the topology of manifolds is described by homology and cohomology.
For instance, one of invariances of manifolds, the Euler number $\chi(M) = \sum_{i=0}^n (-1)^i \beta_i$ can be calculated by the Betti number $\beta_k = \dim H_k(M)$, 
where the dimension of the homology $H_k(M)$ means the number of the $k$-dimensional cycles in $M$, the cycle is a chain consisting of $k$-simplexes without boundary, and simplexes are elemental figures for dividing a whole figure, e.g., points, line segments, triangles, and higher-dimensional counterparts.


\begin{remark}[\cite{b_SMor1}~]\label{Rem41}
On an orientable compact manifold, the following relations hold:
\begin{enumerate}
\item Let $H_k(M)$ be a singular homology group with $\mbb{R}$ coefficients that is derived from a triangulation of $M$ given by the union of simplicial complexes $K$, i.e., $H_k(M) \cong H_k(K)$.
Hence, the singular homology $H_k(M)$ directly represents the topology of $M$. 
\item The dual of singular homology $H_k(M)$ is the singular cohomology $H^k(M)$, and this correspondence can be described by the isomorphism $H^k(M) = (H_k(M))^\ast$.
\item According to the De Rham theorem, the isomorphism $H^k(M) \cong \HDR^k(M)$ exists, i.e., $H_k(M) \cong (\HDR^k(M))^\ast$.
\item By the Hodge theorem, the identification $\varOmega^k_H(M) \cong \HDR^k(M)$ is given, where $\varOmega^i_H(M)$ is the space of harmonic forms on $M$ such that $d \omega_H = 0$ and $\delta \omega_H = 0$ for any form $\omega_H \in \varOmega^i(M)$.
\item Furthermore, if a manifold is closed, i.e., compact without boundary, the isomorphism $H_k(M) \cong \HDR^{n-k}(M)$ is obtained from the Poincar\'{e} duality $\HDR^{n-k}(M) \cong (\HDR^{k}(M))^\ast$. As a result, $H_k(M) \cong \varOmega^{n-k}_H(M)$.
\end{enumerate}
\end{remark}


From the above discussion, one might immediately guess the followings:
\begin{itemize}
\item The Stokes-Dirac structures on general manifolds should include harmonic forms affected by a non-trivial topology.
\item Moreover, the harmonic forms may correspond with flows of vector fields or differential forms.
\item As a further consideration, one of important properties as a port-representation, port-interconnections may change the topology of system domains.
\end{itemize}
These are actually true, and answers to the first and second questions will be shown in the following sections.
Consequently, a global port-interconnection and decomposition for preserving information regarding a global energy flow on a whole domain is derived from harmonic forms, and it will be possible that this concept is applied to attaching extra domains with dynamics for a global energy flow shaping.


\subsection{Stokes-Dirac strictures on manifolds with boundary}

 We shall first recall the Stokes-Dirac structure.
Let $N$ be an $n$-dimensional Riemannian manifold.
Consider a compact oriented subdomain $\Z \subset N$ with a boundary $\pd \Z$.
Now, we assume that $\Z$ is a $\pd$-manifold (see \textit{Definition\;\ref{PDmfd}}).
The purpose of this assumption is to subdivide fields of vectors and differential forms into two types, i.e., normal and tangential components that are related with two different harmonic forms (see \textit{Definition\;\ref{DNfields}}). Note that this change doesn't affect the Stokes-Dirac structure itself, because the restriction $\omega|_{\Z}$ (and $\omega|_{\pd \Z}$) for differential forms $\omega \in \varOmega^k(N)$ can be treated in the same way of the conventional case~\cite{j_ASch1, b_SMor1}; therefore, we denote $\omega|_{\Z}$ by $\omega$ simply.
The Stokes-Dirac structure on $\Z$ can be implicitly described in the following~\cite{j_GNis1}:
\begin{align}
 \left\{
 \begin{array}{l}
  f^p_E = (-1)^r de^q_D \in \varOmega^p_E(\Z), \\[1mm]
  e^p_D = \Hodge f^p_E \in \varOmega^{q-1}_D(\Z), \\[1mm]
  f^q_E = de^p_D \in \varOmega^q_E(\Z), \\[1mm]
  e^q_D = \Hodge f^q_E \in \varOmega^{p-1}_D(\Z)
 \end{array}
 \right. \label{SD}
\end{align}
with the boundary port
\begin{align}
 f_b = e^p_D|_{\pd \Z}, \quad e_b = (-1)^p e^q_D|_{\pd \Z},
\end{align}
where $r = pq + 1$, we have defined the spaces of exact and co-exact forms as
\begin{align}
 \varOmega^k_E(\Z) &= \left\{ d\alpha \mid \alpha \in \varOmega^{k-1}(\Z) \right\}, \\
 \varOmega^k_D(\Z) &= \left\{ \delta \beta \mid \beta \in \varOmega^{k+1}(\Z) \right\},
\end{align}
and $\delta = (-1)^{nk+n+1}\Hodge d \Hodge\colon \varOmega^k(\Z) \to \varOmega^{k-1}(\Z)$ is the co-differential operator that is the adjoint of $d$ in the sense of the pairing $\langle\omega,\eta\rangle = \int_M \omega \wedge \Hodge \eta$, and $\Hodge\colon \varOmega^k(M) \to \varOmega^{n-k}(M)$ is the Hodge star operator.
Then, the power balance of the Hamiltonian $\mscr{H}$ of a given system is given as
\begin{align}
 \dfrac{d}{dt}\mscr{H} = - \int_\Z \left( e^p \wedge f^p + e^q \wedge f^q \right) = \int_{\pd \Z} e^b \wedge f^b. \label{PBeq1}
\end{align}
The correspondence between variables in (\ref{PBeq1}) and a Hamiltonian is ignored here, and we concentrate our interest on its geometrical property.


\subsection{Stokes-Dirac complex and cohomology}

The spaces used in the definition of the Stokes-Dirac structure are related with each other as the following diagram.


\begin{definition}
 The relation between the spaces of the variables in the Stokes-Dirac structure (\ref{SD}) on $\Z$ can be illustrated by the following diagram~\cite{j_GNis1}:
\begin{equation}
\hs{-2}
\begin{xyczero}
\xy 
\UseTips
\xymatrix @R=1.5pc @C=1.5pc {
& 0 \ar[r]^{d} & \varOmega^{p-1}_\SD \ar@{<->}[d]_{\ast} \ar[r]^{d} & \varOmega^{p} \ar@{<->}[d]_{\ast} \ar[r]^{d} & \varOmega^{p+1}_\SE \ar[r]^{d} & 0 \\
0 & \ar[l]_{d} \varOmega^{q+1}_\SE & \ar[l]_{d} \varOmega^{q} & \ar[l]_{d} \varOmega^{q-1}_\SD & \ar[l]_{d} 0 & \ .
}
\endxy
\end{xyczero}\label{bdig1}
\hs{-3}
\end{equation}
 Then, we call the above two parallel sequences that are equivalent to short exact sequences of the de Rham complex \emph{the Stokes-Dirac complex}:
\begin{align}
\left\{
\left(\varOmega^{p}(\Z), d \right),
\left(\varOmega^{q}(\Z), d \right) \label{SDcomp1}
\right\}.
\end{align}
%
\end{definition}


Next, we define the cohomologies of the Stokes-Dirac complex 
are isomorphic to the spaces of harmonic forms, i.e., we can know the availability of harmonic fields by checking the cohomologies.


\begin{definition}
 We define the cohomology $\HSD(\Z)$ of the Stokes-Dirac complex by the set of the de Rham cohomologies $\{ \HDRed{q}{\Z}, \HDRde{p}{\Z}, \HDRed{p}{\Z}, \HDRde{q}{\Z} \}$
for the spaces $\{ \varOmega^{p-1}(\Z), \varOmega^{p}(\Z), \varOmega^{q-1}(\Z), \varOmega^{q}(\Z) \}$ used for defining $\{ e^q, f^p, e^p, f^q \}$, respectively.
\end{definition}


\begin{proposition}
 The cohomology $\HSD(\Z)$ of the Stokes-Dirac complex on a compact $\pd$-manifold $M$ is represented by $\{ \mcal{H}^q_\SN(\Z), \mcal{H}^p_\ST(\Z), \mcal{H}^p_\SN(\Z), \mcal{H}^q_\ST(\Z) \}$.
\end{proposition}

\begin{proof}
 We have used the isomorphisms in \textit{Theorem\;\ref{HodgeIso}} and \textit{Corollary\;\ref{HodgeIsoCor1}}, where $p-1 = n-q$ and $q-1 = n-p$.
\end{proof}


\subsection{Harmonic forms in Stokes-Dirac structures}

In this section, the relation between the standard Stokes-Dirac structure and harmonic forms is clarified.

The following fact is derived from a generalized Hodge decomposition for manifolds with boundary.


\begin{lemma}\label{HMFdcomp1}
 According to Hodge-Morrey-Friedrichs decomposition (\textit{Theorems\;\ref{HMDcomp} and \ref{FDcomp} }), a differential $k$-form $\omega \in \varOmega^k(\Z)$ on a $\pd$-manifold $M$ has the unique splitting 
 \begin{align}
  \omega = d \alpha + \delta \beta + \delta \gamma + \lambda,
 \end{align}
 where $d\alpha \in \mcal{E}^k(\Z)$, $\delta \beta \in \mcal{C}^k(\Z)$, $\delta \gamma \in \mcal{H}^k_{C}(\Z)$, and $\lambda \in \mcal{H}^k_\ST(\Z)$.
\end{lemma}


The following inner product is equivalent to the pairing between Hodge dual differential forms, $\omega \in \varOmega^i(\Z)$ and $\eta \in \varOmega^{n-i}(\Z)$, used for defining effort and flow variables in distributed port-Hamiltonian systems.


\begin{lemma}
 The following transformation of the inner product is derived from Green formula~(\ref{Green}):
 \begin{align}
  &\langle \langle f^i, (-1)^{s} \Hodge e^i \rangle \rangle = \int_{M} d e^j \wedge e^i 
  = \int_{M} d e^j \wedge \Hodge f^i \nt \\
  &\quad = (-1)^{i} \int_{M} e^j \wedge d \Hodge f^i  + \int_{\pd M} \tangent e^j \wedge \tangent \Hodge f^i \nt \\
  &\quad = (-1)^{i} \int_{M} e^j \wedge d \Hodge f^i  + \int_{\pd M} \tangent e^j \wedge \Hodge \normal f^i \label{IPair1}
 \end{align}
for $\{i, j\} \in \{ \{p, q\}, \{q, p\} \}$, where $s = i(n-i) = i(j-1)$.
\end{lemma}

\begin{proof}
 We have used the relation in \textit{Proposition\;\ref{Prop21}}.
\end{proof}



\begin{theorem}
 The boundary term in the power balance equation (\ref{PBeq1}) is extended that on a $\pd$-manifold $M$ as follows:
 \begin{align}
  \int_{\pd \Z} e^b \wedge f^b = \int_{\pd M} \tangent e^j \wedge \tangent \Hodge d \alpha^i +
\int_{\pd M} \tangent e^j \wedge \Hodge \normal \lambda^i, \label{PBeq2}
 \end{align}
 where $d\alpha^i \in \mcal{E}^k(M)$, and $\lambda^i \in \mcal{H}^k_\ST(M)$.
\end{theorem}

\begin{proof}
Because $f^i$ is closed: $df^i = 0$, $\delta \beta^i = \delta \gamma^i = 0$ in the form $f^i = d \alpha^i + \delta \beta^i + \delta \gamma^i + \lambda^i$ obtained from the decomposition (\ref{HMFdcomp1}).
 The formula (\ref{IPair1}) admits the expression $f^i = d \alpha^i + \lambda^i$.
 Indeed, $\alpha^i = e^j$.
 The sum of two appropriate inner products (\ref{IPair1}) yield the boundary term in the last equation in the same manner of the proof of the standard Stokes-Dirac structure~\cite{j_ASch1}.
\end{proof}


 Then, the first term of the right-side in (\ref{PBeq2}) corresponds to the conventional boundary energy flow, the second is the new boundary energy related with the topology of $M$.
 This result is justified as following known fact.


\begin{lemma}[{\cite[p. 127]{b_GSch1}}]\label{Lem52}
 Let $M$ be a $\pd$-manifold. 
 Consider the problem of finding a solution $e^j \in \varOmega^{k-1}(M)$ of the equations
\begin{align}
 f^i = d e^j \ \textrm{\ on\ } M, \qquad e^j|_{\pd M} = \psi|_{\pd M} \ \textrm{\ on\ } \pd M
\end{align}
for given $f^i \in \varOmega^{k}(M)$ and $\psi \in \varOmega^{k-1}(M)|_{\pd M}$.
 This problem is solvable, if and only if 
$f^i$ and $\psi$ satisfy the integrability condition
 \begin{align}
  d f^i &= 0, \quad \tangent f^i = \tangent d \psi, \\
  \langle \langle f^i, \lambda^i \rangle \rangle &= \int_{\pd M} \tangent \psi \wedge \Hodge \normal \lambda^i \quad \forall \lambda^i \in \mcal{H}^{k}_\ST(M). \label{GPintA}
 \end{align}
\end{lemma}


\begin{proposition}
 Consider the Stokes-Dirac structure on a $\pd$-manifold $\Z$ for $f^p \in \varOmega^{p}(\Z)$, $e^p \in \varOmega^{q-1}(\Z)$, $f^q \in \varOmega^{q}(\Z)$, and $e^q \in \varOmega^{p-1}(\Z)$. Then, there exists the harmonic forms $\lambda^p \in \mcal{H}^{p}_\ST(\Z)$ and $\lambda^q \in \mcal{H}^{q}_\ST(\Z)$ satisfying
 \begin{align}
  \tangent f^p &= (-1)^r d \tangent e^q, & 
  \langle \langle f^p, \lambda^p \rangle \rangle &= \int_{\pd M} \tangent e^q \wedge \Hodge \normal \lambda^p, \label{GPintA2} \\
  \tangent f^q &= d \tangent e^p, & 
  \langle \langle f^q, \lambda^q \rangle \rangle &= \int_{\pd M} \tangent e^p \wedge \Hodge \normal \lambda^q. \label{GPintB2}
 \end{align}
 We call the boundary integrations (\ref{GPintA2}) and (\ref{GPintB2}) \emph{harmonic boundary energy flows}, and $\lambda^p$ and $\lambda^q$ \emph{harmonic boundary energy variables}.
\end{proposition}

\begin{proof}
 By substituting $\phi = e^j$ describing the free boundary condition to \textit{Lemma\;\ref{Lem52}}, we can get the results.
\end{proof}


\begin{theorem}
 The cohomology $\HSD(\Z)$ of the Stokes-Dirac structure on a $\pd$-manifold $\Z$ is affected by the topology of $\Z$.
\end{theorem}

\begin{proof}
 By the Hodge isomorphism \textit{Theorem\;\ref{HodgeIso}}, $\HDR^k(M,\delta) \cong \mcal{H}^k_\ST(M)$.
Moreover, by \textit{Corollary\;\ref{HodgeIsoCor1}}, we have $\HDR^k(M,d) \cong \HDR^{n-k}(M,\delta)$.
The de Rham cohomology $\HDR^k(M,d)$ is related with the singular homology group representing the topology of $\Z$.
On the other hand, it is known that $\HDR^k(M,\delta)$ is isomorphic to the relative de Rham cohomology and the relative homology~\cite{b_GSch1, b_GNis1} that is an alternate homological classification of manifolds with boundary. Indeed, because $\HDR^k(M,d) \cong \HDR^{n-k}(M,\delta)$ as we have seen, the same result can be obtained from the both.
\end{proof}

\section{Examples}\label{Sec5}


\memo{
\begin{figure}[h]
 \mbox{
 \begin{minipage}[b][30mm][c]{0.95\hsize}
  \centering
  \includegraphics[width=0.3\columnwidth,clip]{./fig/bloodflow1.eps}
 \end{minipage}
 }
  \caption{Blood flow~\cite{x}}
  \label{fig2}
\end{figure}
} 

The homology of $M$ can be detected by harmonic forms as we have seen before.
From the viewpoint of topology, the homology group $H_i(M)$ can be interpreted, e.g., in the case $n=3$ as follows:
\begin{itemize}
\item $H_0(M)$ $\cdots$ The vector space generated by equivalence classes of points in $M$ such that two points are equivalent if there exists a path connecting the points in $M$. 
$\dim H_0(M)$ is equivalent to the number of components of $M$. 
Note that $H_0(M) \cong \mbb{R}$ if $M$ is connected, then the element of $H_0(M)$ is a constant function.
\item $H_1(M)$ $\cdots$ The vector space generated by equivalence classes of oriented loops in $M$ such that two loops are equivalent if their difference is the boundary of an oriented surface in $M$. $\dim H_1(M)$ is equivalent to the number of total genus of $\pd M$, where a genus means the number of holes of closed surfaces.
\item $H_2(M)$ $\cdots$ The vector space generated by equivalence surfaces in $M$ such that two surfaces are equivalent if their difference is the boundary of some oriented subregion of $M$.
$\dim H_2(M)$ is equivalent to the number of the difference between components of $\pd M$ and those of $M$.
\item $H_3(M)$ $\cdots$ $\dim H_3(M)$ is always $0$.
\end{itemize}

\memo{
On the other hand, the dual space of $H_{k}(M)$ is $H_{n-k}(M,\pd M)$, where $H_{k}(M,\pd M)$ is called \emph{the $k$-th relative homology of $Z$ modulo $\pd M$}.
In $n = 3$, the relative homology of $Z$ modulo $\pd M$ consists of the following vector spaces with real coefficients:
\begin{itemize}
\item $H_0(M,\pd M)$ $\cdots$ $\dim H_0(M)$ is always $0$. 
\item $H_1(M,\pd M)$ $\cdots$ The vector space is generated by such equivalence classes of oriented paths whose endpoints lie on $\pd M$ as two such paths are equivalent if their difference (possibly paths on $\pd M$) is the boundary of an oriented surface in $M$. 
\item $H_2(M,\pd M)$ $\cdots$ The vector space is generated by such equivalence classes of oriented surface whose boundaries lie on $\pd M$ as two such surfaces are equivalent if their difference (possibly portions of $\pd M$) is the boundary of some oriented subregion of $M$. 
\item $H_3(M,\pd M)$ $\cdots$ The vector space has the oriented components of $M$ as a basis. Thus, $\dim H_3(M,\pd M)$ is the number of components of the subregions of $M$ whose boundaries lie on $\pd M$. Note that $H_3(M,\pd M) \cong \mbb{R}$ for a connected $M$ and the element of $H_3(M,\pd M)$ is a constant function.
\end{itemize}
Hence, $H_k(M,\pd M) \cong H_{3-k}(M)$ for $0 \leq k \leq 3$.
} 

Because vector fields can be identified with $1$-forms on manifolds, the above decomposition of differential forms affects the vector fields.


\begin{theorem}[~\cite{j_JCan1}~]
Let $M$ be a compact domain with a smooth boundary $\pd M$ in three-dimensional space.
Let $\VF(M)$ be the infinite dimensional vector space of all vector fields in $M$.
Consider $L^2$ inner product $\langle \GVS, \GWS \rangle = \int_M \GVS \cdot \GWS\,dx$ for any $\GVS, \GWS \in \VF(M)$, where $dx$ is the volume form on $M$. 
The space $\VF(M)$ is the direct sum of the following mutually orthogonal subspaces:
\begin{align}
 \VF(M) &= \VK(M) \varoplus \VG(M),
\end{align}
where $\GVS \in \VF(M)$, $\varphi \in C^\infty(M)$, 
\begin{align}
 \VK(M) &= \left\{ \GVS \in \VF(M)\mid \Div \GVS = 0,\ \langle \GVS, \vl{n} \rangle = 0
\right\}, \\
 \VG(M) &= \left\{ \GVS \in \VF(M)\mid \GVS = \Grad \varphi \right\},
\end{align}
which are called \emph{knots} and \emph{gradients}, respectively, and $\vl{n}$ means unit vector fields normal to $\pd M$.
Furthermore, the subspaces
%
%
%
\begin{align}
 \HK(M) &= \left\{ \GVS \in \VK(M)\mid \Curl \GV = \vl{0} \right\}, \\
 \HG(M) &= \left\{ \GVS \in \VG(M)\mid \Div \GVS = 0,\ \right. \nt \\
 &\quad \left. \varphi \text{ is locally constant on } \pd M \right\}, 
\end{align}
which are respectively called 
\emph{harmonic knots} and
\emph{harmonic gradients},
directly relate to the topology of $M$ as follows:
\begin{align}
 \dim H_1(M) &= \dim \HK(M), \quad \dim H_2(M) = \dim \HG(M).
\end{align}

\end{theorem}


\begin{example}
The homology group of the two-dimensional sphere $M = S^2$ (see the left of Fig.\ref{fig1}) consists of
$H_2 \cong \mbb{R}$, $H_1 \cong 0$, and $H_0 \cong \mbb{R}$.
There is no $\GVS \in \HK(M)$ in $S^2$, because $\dim H_1(M) = \dim \HK(M) = 0$.
However, there exists $\GVS \in \HG(M)$ on $S^2$, because $\dim H_2(M) = \dim \HG(M) = 1$.
Then, $\HG(M)$ means a radiational vector field flowing from an internal point of the sphere.

The homology group of the two-dimensional torus $M = T^2$  (see the right of Fig.\ref{fig1}) consists of $H_2 \cong \mbb{R}$, $H_1 \cong \mbb{R} \varoplus \mbb{R}$, and $H_0 \cong \mbb{R}$.
Thus, $\dim \HK(M) = 2$, and $\dim \HG(M) = 1$.
Then, $\HK(M)$ means circulative vector fields around loops that are non-contractible.
The difference between the above two cases is $H_1$.

Let us consider $M = T^2$ as a $\pd$-manfold.
According to the Lefschetz duality~\cite[pp. 105]{b_GSch1}, $H^k(\Z, \delta) \cong H_{n-k}(\Z)$.
Therefore, $H_{n-k}(\Z) \cong \mcal{H}^k_\ST(M)$.
The degree of Stokes-Dirac structures on these two dimensional domain may be $(p-1,\,p,\,q-1,\,q) \in \{(0,\, 1,\, 2,\, 3), (1,\, 2,\, 1,\, 2), (2,\, 3,\, 0,\, 1)\}$.
Hence, in the case of $p = 2$ and $q = 2$, the extra terms in (\ref{GPintA2}) and (\ref{GPintB2}) corresponding $\HK(M)$ of $H_1$ appear.

\begin{figure}[h]
 \mbox{
 \begin{minipage}[b][30mm][c]{0.95\hsize}
  \centering
  \reflectbox{\includegraphics[width=0.5\columnwidth,clip,angle=0,origin=c]{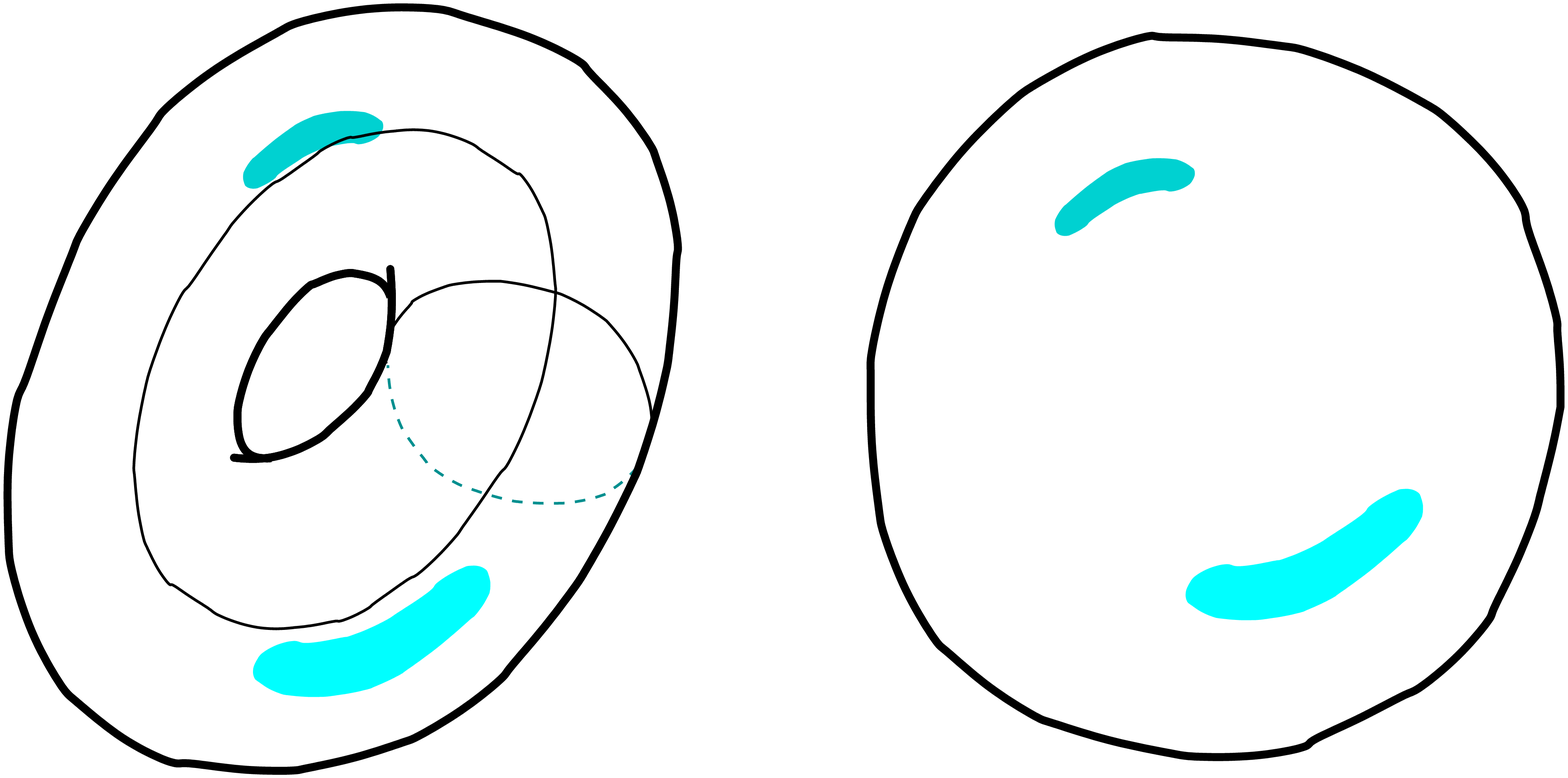}}
 \end{minipage}
 }
  \caption{Two-dimensional sphere and torus}
  \label{fig1}
\end{figure}
\end{example}



\section*{Appendix}
\setcounter{subsection}{0}

\subsection{Vector fields on boundary}\label{VFonB}

A metric on a $\pd$-manifold is defined as a smooth map $g\colon TM \times TM \to \mbb{R}$ such that $g|_p\colon T_pM \times T_pM \to \mbb{R}$ is symmetric, bilinear and positive definite for all $p \in M$.
Let $(M,g)$ be a Riemannian $\pd$-manifold.
Consider a vector field $E_i \in \Gamma(TU)$ such that $g(E_i, E_j)|_{p} = \delta_{ij}$ for any $p \in U \subset M$, where $1 \leq i,j \leq n-1$ and the space of all smooth vector field is denoted by $\Gamma$. The tuple $(E_1, \cdots, E_n)$ is called a \emph{$g$-orthonormal frame}.

According to the collar theorem~\cite{b_GSch1}, a normal vector field $\mcal{N}\colon \pd M \to TM|_{\pd M}$ defined on $\pd M$ can be smoothly extended to a vector field $\widehat{\mcal{N}}$ on a neighborhood $U$ of the boundary $\pd M$.
Then, by choosing $\widetilde{\mcal{N}} = \widehat{\mcal{N}}/|\widehat{\mcal{N}}|$, 
any vector field $X \in \Gamma(TM)$ in $U$ can be uniquely decomposed into its tangential component $X^\parallel$ and its normal component $X^\perp$, i.e., $X = X^\parallel + X^\perp$, where $X^\parallel = g(X, \widetilde{\mcal{N}}) \widetilde{\mcal{N}}$, and  $g(X^\perp, \widetilde{\mcal{N}}) = 0$.
This construction can determine a $g$-orthonormal frame $\mcal{F} = (\widetilde{\mcal{N}}, E_1, \cdots, E_{n-1})$ on any sufficiently small neighborhood $U$ intersecting $\pd M$, where $\widetilde{\mcal{N}}|_{\pd M} = \mcal{N}$ and $E_i|_{\pd M} \in T \pd M$.


\subsection{Sobolev space of differential forms}\label{Sobolev}


\memo{

\begin{definition}
もし，$\langle\,,\, \rangle_{\mbb{F}}$が滑らかに$p \in M$に依存するならば，ファイバー計量は，正定，対称，双線形な写像$\langle\,,\, \rangle_{\mbb{F}}|_p \colon \pi^{-1}(p) \times \pi^{-1}(p) \to \mbb{R}$となる．
与えられたファイバー計量を用いて，その束アトラスから，任意の$p \in U_a$と各$U_a$に対して，$\langle \nu_i, \nu_j \rangle_{\mbb{F}}|_p = \delta_{ij}$となるような，$\Gamma(\mbb{F}|_{U_a})$の\textbf{局所直行枠}$(\nu_1, \cdots, \nu_m)$が得られる．ここで，$1 \leq i,j \leq m$である．
\end{definition}

}


\begin{definition}
 Let $M$ be a $\pd$-manifold. Consider a vector bundle $\mbb{F}$ over $M$ with a fibre metric $\langle \,, \, \rangle_{\mbb{F}}$ and a connection $\nabla$. 
 Let $\{ U_a \}_{a \in A}$ be an open cover of $M$, $\{ \rho_a \}_{a \in A}$ be a subordinated partition unity, and $(E^a_1, \cdots, E^a_n)$ be a family of local frames.
 Then , \emph{$W^{s,p}$-norm on $\Gamma(\mbb{F})$} is defined by
 \begin{align}
  \| \sigma \|^p_{W^{s,p}} &= \sum_{a \in A} \int_M \rho_a | \sigma |^p_{J^s(\mbb{F}|U_a)} \mu, \label{normW}\\
  \| \sigma \|^p_{W^{1,p}} &= \int_M | \sigma |^p_{J^1(\mbb{F})} \mu,
 \end{align}
 where $1 \leq p < \infty$, $s \in \mbb{N}_0$, $\mu \in \varOmega^n(M)$ is the Riemannian volume form on $M$, $| \cdot |^p_{J^s(\mbb{F}|U_a)}\colon \Gamma(\mbb{F}_U) \to C^\infty(U)$ is the fiber norm such that
 \begin{align}
  | \sigma |^2_{J^0(\mbb{F})} &= \langle \sigma, \sigma \rangle_{\mbb{F}}, \nt \\
  | \sigma |^2_{J^s(\mbb{F})} &= | \sigma |^2_{J^{s-1}(\mbb{F})} + \sum_{1 \leq j \leq n} | \nabla_{E_j} \sigma |^2_{J^{s-1}(\mbb{F})},
 \end{align}
 and the map $\nabla_X = \nabla(X, \cdot)$ for a fixed $X \in \Gamma(TM)$ that is called \emph{the covariant derivative in the direction $X$} is induced from the connection $\nabla$. 
\end{definition}

 The space of smooth compactly supported sections is denoted by $\Gamma_c(\mbb{F})$.
 \emph{The Sobolev space} $W^{s,p}\Gamma(\mbb{F})$ is defined as the completion of $\Gamma_c(\mbb{F})$ with respect to the norm (\ref{normW}).
If $M$ is compact, then $\Gamma_c(\mbb{F}) = \Gamma(\mbb{F})$.


Consider the exterior $k$-form bundle $\varLambda^k(M)$ over a Riemannian $\pd$-manifold $M$ as $\mbb{F}$ in the above definition.
In this case, the space $\Gamma_c(\varLambda^k(M))$ of smooth compactly supported sections of $\varLambda^k(M)$ is that of compactly supported differential forms on $M$.
 This space is equipped with a $L^2$-inner product $\left<\left< \omega, \eta \right>\right> = \int_M \omega \wedge \Hodge \eta$, and the corresponding fiber metric on $\varLambda^k(M)$ is 
 $\langle \omega, \eta \rangle_{\varLambda^k} \mu_M = \omega \wedge \Hodge \eta$.

\section*{Acknowledgement}


The authors would like to thank Professor A. J. van der Schaft for fruitful discussions.

\end{document}